\documentclass{amsart}

\newtheorem{theorem}{Theorem}[section]
\newtheorem{lemma}[theorem]{Lemma}
\newtheorem{corollary}[theorem]{Corollary}

\numberwithin{equation}{section}

\newtheorem{proposition}[theorem]{Proposition}

\theoremstyle{definition}

\theoremstyle{remark}
\newtheorem{remark}[theorem]{Remark}

\usepackage{epsfig,amscd,amssymb,amsxtra,amsmath,amsthm}
\usepackage{url}

\begin{document}
\title[Reducing a class of induced representations]{Composition series of a class of induced representations built on discrete series}
\author{Igor Ciganovi\'{c}}
\email{igor.ciganovic@math.hr}
\address[Igor Ciganovi\'{c}]{Department of Mathematics\\University  
	                         of Zagreb\\
                             Bijeni\v{c}ka cesta 30\\ 10 000 Zagreb\\ Croatia}
\thanks{This work has been fully supported by Croatian Science Foundation under the
	project IP-2018-01-3628.}
\begin{abstract} 
We have determined composition series of a class of induced representations appearing in M{\oe}glin-Tadi\'c classification of discrete series. The result is further used to determine composition series of certain representations induced from Langlads quotients.
This should provide more information on decomposing standard representations as well as Jacquet modules of discrete series, which has applications in automorphic forms. 
\end{abstract}
\maketitle
\section{Introduction}
\label{intro}
As standard representations are used to classify irreducible representations, determining their composition series is an important, but hard problem.
Furthermore, a certain subclass of standard representations is an integral
part of M{\oe}glin-Tadi\'c classification of discrete series.
However, an attempt to decompose any member of that subclass, using intertwining operators, requires decomposition of  even more special subclass, also part of M{\oe}glin-Tadi\'c classification. In this paper we determine composition series of that special subclass.
As it is unbounded in the number of essentially square integrable representations of general linear groups, we belive that the decomposition provides a valuable information for general decomposition of standard representations. 
On the other hand, as strongly positive discrete series play fundamental role in M{\oe}glin-Tadi\'c classification of discrete series, and their Jacquet modules are determined by I. Mati\'c, our decomposition  gives a direct way of analysing Jacquet modules of a large class of discrete series without constructing them step by step.

To describe our results we introduce some notation. 
Fix a local non-archimedean field $F$ of characteristic different from two. 
If $\delta$ is an essentially square integrable representation of $GL(m_{\delta},F)$ (this defines $m_{\delta}$), where $m_{\delta}>0$, then there exists
an irreducible cuspidal unitary representation of $\rho$ of $GL(m_{\rho},F)$ (this defines $m_{\rho}$) and $x,y\in\mathbb{R}$, such that $y-x+1\in\mathbb{Z}_{\geq 0}$ and $\delta$ is a unique irreducible subrepresentation of the parabolically induced representation $\nu^y\rho\times \cdots \times \nu^x\rho$.
The set 
$\Delta=[\nu^x \rho,\nu^y \rho]=\{\nu^x \rho,...,\nu^y \rho\}$ is called segment.
Also, we denote $\delta(\Delta)=\delta$ and $e(\Delta)=\frac{x+y}{2}$.

Let $G_n$ be a symplectic or (full) orthogonal group    
having split rank $n$. Let $\tau$ be an irreducible tempered representation of $G_n$
and  $\Delta_1,...,\Delta_k$, sequence of segments such that
$e(\Delta_1)\geq \cdots \geq e(\Delta_k)>0$. The parabolically induced representation
\begin{equation} \label{langlands}
\delta(\Delta_{1})\times\cdots\times \delta(\Delta_k)\rtimes \tau
\end{equation}
is called a standard representation. 
It has a unique irreducible quotient, called the Langlands quotient. By Langlands classification, all irreducible representations can be described as Langlands quotients, with trivial data for irreducible tempered representations. 
We call a discrete series $\sigma_{sp}$ of $G_n$ strongly positive if it is cuspidal or for every embedding of form
\[
\sigma_{sp} \hookrightarrow \nu^{x_1}\rho_1\times \cdots \times \nu^{x_l}\rho_{l} \rtimes \sigma_{cusp}
\]
where $x_i\in \mathbb{R}$, $\rho_i$  is  an irreducible unitary cuspidal representation $GL(m_{\rho_i},F)$ for all 
$i=1,...,l\in \mathbb{Z}_{>0}$  and  an irreducible representation
$\sigma_{cusp}$ is a cuspidal representation of $G_{n'}$, for some $n'$, 
we have
\[
x_1>0,\ldots,x_{l}>0.
\]
By
M{\oe}glin-Tadi\'c classification,  discrete series, that are not strongly positive, can be described as subrepresentations of some induced  representations,
of the form
\begin{equation} \label{mt}
\delta(\Delta_{1})\times\cdots\times \delta(\Delta_k)\rtimes \sigma_{sp}
\end{equation}
with certain conditions on segments and
an additional parameter - $\epsilon$ function.

The main result of this paper is composition series of representation \eqref{mt} with respect to an additional condition: for all $i\neq j$ 
induced representations
$\delta(\Delta_{i}) \times\delta(\Delta_j)$ and
$\delta(\widetilde{\Delta_i}) \times \delta(\Delta_{j})$ 
are irreducible. This condition is not unnatural and it has been considered in \cite{tadic-family}, with also $\sigma_{sp}$ being cuspidal.
We have the following:
\begin{theorem} \label{prikaz}
Let $\sigma$ be a discrete series described by    
M\oe glin Tadi\'c classification
as a subrepresentation of the induced representation 
\[
\sigma \hookrightarrow
\delta(\Delta_{1})\times\cdots\times \delta(\Delta_k)\rtimes \sigma_{sp},
\]
where $\sigma_{sp}$ is a strongly positive discrete series and $\Delta_{1},\ldots ,\Delta_{k}$ are segments.

Assume that for all $i\neq j$ induced representations
$\delta(\Delta_{i}) \times\delta(\Delta_j)$ and
$\delta(\widetilde{\Delta_i}) \times \delta(\Delta_{j})$ 
are irreducible.
Then the induced representation is multiplicity one. All irreducible subrepresentations, there are $2^k$ of them, are discrete series extensions of
$\sigma_{sp}$.
 Denote $S=\{1,\ldots,k\}$. 
 In the appropriate Grothendieck group we have
 
\[
	\delta(\Delta_{1})\times\cdots\times \delta(\Delta_k)\rtimes \sigma_{sp}
	=
	\sum_{X\subseteq S} \ \ \ 
	\sum_{ \sigma'\hookrightarrow \prod_{j\in S\setminus X} \delta(\Delta_j)\rtimes \sigma_{sp}}
	Lang(\prod_{i\in X} \delta(\Delta_i)\rtimes \sigma')
\]
where  $\sigma'$ is used to denote an irreducible subrepresentation, also being a discrete series extension of $\sigma_{sp}$.
For every integer 
\(0\leq l \leq k\), let $V_l$ be an image of the intertwining operator
\[
\bigoplus_{\substack{X\subseteq S,\\ \text{card}(X)=l}} \ \ 
\bigoplus_{ \sigma'\hookrightarrow \prod_{j\in S\setminus X} \delta(\Delta_j)\rtimes \sigma} \ 
\left(\prod_{i\in X} \delta(\Delta_i)\rtimes \sigma' \right)
\longrightarrow
\prod_{i\in S} \delta(\Delta_i)\rtimes \sigma
\]
given by $(x_1,x_2,\ldots) \mapsto x_1+x_2+\cdots$. Thus 
$V:=V_{k}=\prod_{i\in S} \delta(\Delta_i)\rtimes \sigma$ and  
$V_0$ is a direct sum of  irreducible subrepresentations of 
$V_k$. 
We have a filtration
$\{0\}=V_{-1}\subseteq V_0\subseteq \ldots  \subseteq V_k$, 
where for every integer $0\leq l \leq k$ we have
\begin{equation*} 
	V_l/V_{l-1}\cong
	\bigoplus_{\substack{X\subseteq S,\\ \text{card}(X)=l}} \ \ 
	\bigoplus_{ \sigma'\hookrightarrow \prod_{j\in S\setminus X} \delta(\Delta_j)\rtimes \sigma} \ 
	Lang(\prod_{i\in X} \delta(\Delta_i)\rtimes \sigma' ).
\end{equation*}
The induced representation has length  \
$
\sum_{l\geq 0}\binom{k}{l}2^{k-l}=(1+2)^k=3^k.
$
\end{theorem}

We note that this theorem is a special case of the class of induced representations that we actually decomposed in Theorem \ref{glavni}, see Corollary  \ref{mt-dekompozicija}. Also 
the case of an induction from two segments and $\sigma_{sp}$ cuspidal is solved in \cite{ciganovic}.

For the convenience, using assumptions as in Theorem \ref{prikaz}, we also give here an interesting consequence.
\begin{corollary}
Let 
${\pi \leq 
\delta(\Delta_{1})\times\cdots\times \delta(\Delta_k)\rtimes \sigma_{sp} }
$ 
be any irreducible subquotient. 
Taking possible contragredients of segments in  ${\{\Delta_i : i\in S\}}$, there exist segments 
$\{\Delta'_i : i\in S\}$, such that
\[
\pi 
\hookrightarrow
\delta(\Delta'_{1})\times\cdots\times \delta(\Delta'_k)\rtimes \sigma_{sp}.
\]
\end{corollary}

\section{Preliminaries}
\label{sec-1}

Fix  a local non-archimedean field $F$ of characteristic different from two. As in \cite{tadic-diskretne} let $V_n$, $n\geq 0$ be a tower of symplectic or orthogonal non-degenerate $F$ vector spaces  where $n$ is the Witt index. We denote by $G_n$ the group of isometries of $V_n$, by Irr $G_n$ set of irreducible representations of $G_n$ and by
Irr$'$ the set $\cup_{n\in \mathbb{N} } \text{ Irr } G_n$.  
Group  $G_n$ has split rank $n$.  
Also, we fix a set of standard parabolic subgroups in the usual way.  
Standard parabolic proper subgroups of $G_n$ are in bijection
with the set of ordered  partitions of positive integers $m\leq n$. Given positive integers $n_1,...,n_k$ such that $m=n_1+\cdots+n_k\leq n$ the corresponding standard parabolic subgroup $P_s$, $s={(n_1,...,n_k)}$ has the Levi factor $M_s$ isomorphic to
\(
GL(n_1,F)\times \cdots \times GL(n_k,F)  \times G_{n-m}.
\)
So an irreducible representation $\pi$ of $M_s$ can be written as   
$\delta_1\otimes\cdots\otimes\delta_k\otimes \tau$ where
$\delta_i$ is an irreducible representation of $GL(n_i,F)$, $i=1,...,k$ and 
$\tau$ an irreducible representation of $G_{n-m}$.
We use the following notation for the normalized parabolic induction
$$
\delta_1\times\cdots\times\delta_k\rtimes \tau= 
\text{Ind}_{M_s}^{G_n}(\delta_1\otimes\cdots\otimes\delta_k\otimes \tau).
$$
If $\sigma$ is a smooth representation of $G_n$ we denote by 
$\text{r}_{s}(\sigma)=\text{r}_{M_s}(\sigma)=\text{Jacq}_{M_s}^{G_n}(\sigma)$
the normalized Jacquet module of $\sigma$. 
If $\text{r}_{M_s}(\sigma)$ is trivial for every proper standard parabolic subgroup $P_s$  then
$\sigma$ is said to be cuspidal.
We have Frobenius reciprocity
\[
\text{Hom}_{G_n}(\sigma,\text{Ind}_{M_s}^{G_n}(\pi))=
\text{Hom}_{M_s}(\text{Jacq}_{M_s}^{G_n}(\sigma),\pi).
\]

We recall some results about representations of general linear groups  from \cite{zelevinsky-ind-repns-II}.
Let $\rho$ be 
an irreducible cuspidal unitary representation of  $GL(m_{\rho},F)$ (this defines $m_{\rho}$) and $x,y\in\mathbb{R}$, such that $y-x+1\in\mathbb{Z}_{\geq 0}$. The set 
$\Delta =[\nu^x \rho,\nu^y \rho]=\{\nu^x \rho,...,\nu^y \rho\}$ is called segment.
The induced representation 
$\nu^y\rho\times \cdots \times \nu^x\rho$ has a unique irreducible subrepresentation. It is essentially square integrable, and we denote it by
$\delta(\Delta)$. We also denote 
$e(\Delta)=\frac{x+y}{2}$.
 For $y-x+1\in\mathbb{Z}_{< 0}$ define $[\nu^x \rho,\nu^y \rho]=\emptyset$ and let $\delta(\emptyset)$ be an irreducible representation of the trivial group.
Further, let  
$\widetilde{\Delta}=[\nu^{-y} \widetilde{\rho},\nu^{-x} \widetilde{\rho}]$ where $\widetilde{\rho }$ denotes the contragredient of $\rho$. We have 
$\delta(\Delta)\widetilde{\ }=\delta(\widetilde{\Delta})$.
If $\delta$ is an essentially square integrable representation of $GL(m_{\delta},F)$,
there exists a segment $\Delta$ such that $\delta=\delta(\Delta)$. Segments 
$\Delta'$ and $\Delta''$ are said to be linked if
$\Delta'\nsubseteq \Delta''$ and $\Delta''\nsubseteq \Delta'$ and 
$\Delta'\cup \Delta''$ is a segment. 
If they are 
linked, the induced representation $\delta(\Delta')\times \delta(\Delta'')$ is of 
length 2 and $\delta(\Delta'\cap \Delta'')\times \delta(\Delta'\cup \Delta'')$ is an irreducible subquotient.
Else $\delta(\Delta')\times \delta(\Delta'') \cong \delta(\Delta'')\times \delta(\Delta')$ is an irreducible representation.

Now we write Tadi\'c formula for computing Jacquet modules. Let $R(G_n)$ be the Grothendieck group of the category of smooth representations of $G_n$ of finite length. It is a free Abelian group generated by classes of irreducible representations of $G_n$.  If $\sigma$ is a smooth representation of a finite length of $G_n$, denote by $\text{s.s.}(\sigma)$ the semisimplification of $\sigma$, that is a sum of classes of composition factors of $\sigma$.
Put $R(G)=\oplus_{n\geq 0} R(G_n)$.
For $\pi_1,\pi_2\in R(G)$ we define $\pi_1 \leq \pi_2$ if  $\pi_2-\pi_1$ is a linear combination of classes of irreducible representations with non-negative coefficients.
 Similarly, let $R(GL)=\oplus_{n\geq 0}R(GL(n,F))$. We have the map 
 $\mu^* : R(G)\rightarrow R(GL) \otimes R(G)$ defined by
 \[
 \mu^*(\sigma)=1\otimes \sigma + \sum_{k=1}^{n} \text{s.s.}(r_{(k)}(\sigma)),\  \sigma \in R(G_n).
 \]
The following result derives from Theorems 5.4 and 6.5 of  \cite{tadic-structure}, see also Section 1.\ of \cite{tadic-diskretne}. They are based on Geometrical  Lemma (2.11 of \cite{bernstein-zelevinsky-ind-repns-I}).
\begin{theorem} Let $\sigma$ be a smooth representation of a finite length of $G_n$, $\rho$ an irreducible unitary cuspidal representation of $GL(m_{\rho},F)$ and  
$x,y\in\mathbb{R}$, such that $y-x+1\in\mathbb{Z}_{\geq 0}$.Then
\begin{equation} \label{komnozenje}
\begin{split}
\mu^*(&\delta([\nu^x\rho,\nu^y\rho])\rtimes \sigma)=
\sum_{\delta'\otimes\sigma'\leq \mu^*(\sigma)}
\sum_{i=0}^{y-x+1} \sum_{j=0}^{i}                                \\
&\delta([\nu^{i-y}\widetilde{\rho},\nu^{-x}\widetilde{\rho}])
\times
\delta([\nu^{y+1-j}\rho,\nu^{y}\rho])\times \delta'
\otimes
\delta([\nu^{y+1-i}\rho,\nu^{y-j}\rho])\rtimes \sigma'
\end{split}
\end{equation}
where $\delta'\otimes\sigma'$ denotes an irreducible subquotient in the appropriate Jacquet module.
\end{theorem}
We also note that in the apropriate Grothendieck group
\begin{equation} \label{kontra}
\delta([\nu^x\rho,\nu^y\rho])\rtimes \sigma=
\delta([\nu^{-y}\widetilde{\rho},\nu^{-x}\widetilde{\rho}])\rtimes \sigma.
\end{equation}

The M{\oe}glin-Tadi\'c  classification of discrete series (\cite{moeglin},\cite{tadic-diskretne}) sets up a bijection between classes of discrete series of $G_n, n\in \mathbb{N}$ and objects called admissible triples. The classification, written under the natural hypothesis, is now
unconditional, see page 3160 of \cite{ba}. We briefly recall the classification. 
Let $\sigma$ be a discrete series of $G_n$ for some $n\in \mathbb{N}$, $\rho$  an irreducible, unitarizable, self-dual cuspidal representation of $GL(m_{\rho},F)$ and $a$ a positive integer. The representation 
 \[
 \delta([\nu^{-(a-1)/2}\rho,\nu^{(a-1)/2}\rho])\rtimes \sigma
 \]
 is irreducible for all $a$ of one parity. For the other parity, the representation reduces except for a finite number of integres and their parity is determined only by $\rho$. We define 
$\text{Jord}(\sigma)$ as a set of all pairs $(a,\rho)$ that form such exceptions.  Also,  let
$\text{Jord}_{\rho}(\sigma)=\{ a\in \mathbb{N}\ |\ (a,\rho) \in \text{Jord} \}$.
Next,  we define a partial cuspidal support of $\sigma$, denoted by  
$\sigma_{cusp}$, 
 as a unique irreducible cuspidal representation of some  $G_{n'}$ such that 
there exists an irreducible representation $\pi$ of $GL(m_{\pi},F)$ with the property 
$\sigma \hookrightarrow \pi \rtimes \sigma_{cusp}$.

Now we define admissible triples. First consider a triple 
$(\text{Jord}, \sigma',\epsilon)$ described as follows. Jord is a finite set, possibly empty, of 
pairs $(a, \rho)$ where $\rho$ is an irreducible self-dual cuspidal representation of $GL(m_{\rho},F)$ and $a$ is a positive integer of an appropriate parity, explained above as the parity of exceptions. Next, $\sigma'$ is an irreducible cuspidal representation of $G_{n'}$ for some $n'\in\mathbb{N}$.
Finally, 
 $\epsilon$ is  a function from a subset of $\text{Jord} \cup (\text{Jord}\times \text{Jord})$ into  $\{\pm1\}$. 
It is defined on a pair $((a,\rho), (a',\rho'))\in \text{Jord}\times \text{Jord}$   if and only if $\rho\cong\rho'$ and $a\neq a'$. 
Further, $\epsilon$ is defined on $(a,\rho)\in \text{Jord}$ if and only if $a$ is even or $a$ is odd and $\rho\rtimes \sigma_{cusp}$ reduces. 
Following must hold:
\begin{itemize}
\item[$\bullet$] 
value of $\epsilon$ on a pair $((a,\rho),(a',\rho))$ is denoted by  
$\epsilon(a,\rho)\epsilon(a',\rho)^{-1}$ and it is equal to the product of 
 $\epsilon(a,\rho)$ and $\epsilon(a',\rho)^{-1}$ if they are defined,
\item[$\bullet$] 
$\epsilon(a,\rho)\epsilon(a'',\rho)^{-1}=
\big(\epsilon(a,\rho)\epsilon(a',\rho)^{-1}\big)
\big( \epsilon(a',\rho)\epsilon(a'',\rho)^{-1} \big)$,
\item[$\bullet$]
$\epsilon(a,\rho)\epsilon(a',\rho)^{-1}=\epsilon(a',\rho)\epsilon(a,\rho)^{-1}$.
\end{itemize}
Triple $(\text{Jord},\sigma',\epsilon)$ is said to be alternated if 
\begin{itemize}
\item
$\epsilon(a,\rho)\epsilon(a_{-},\rho)^{-1}\neq 1$
for all $(a,\rho) \in \text{Jord}$ such that there exists  \\
$(a_{-},\rho)\in \text{Jord}$ %, and also 
\item
for every $\rho$ appearing in Jord there exist an increasing bijection \\
${\Phi_{\rho} : \text{Jord}_{\rho} \rightarrow \text{Jord'}_{\rho}(\sigma_{cusp}) }$ where 
\[
\text{Jord'}_{\rho}(\sigma_{cusp})=
\begin{cases}
\text{Jord}_{\rho}(\sigma_{cusp}) \cup \{ 0\} &\text{if } a \text{ is even and } \epsilon(\text{min}(\text{Jord}_{\rho}),\rho)=1,\\
\text{Jord}_{\rho}(\sigma_{cusp})  &\text{else}.
\end{cases}
\]
\end{itemize}
Triple $(\text{Jord},\sigma',\epsilon)$ is said to be admissible if it can be reduced to an alternated triple in a finite number of steps by removing pairs such that
  $\epsilon(a,\rho)\epsilon(a_{-},\rho)^{-1}= 1$ and accordingly restricting 
the $\epsilon$ function.

Now the classification of discrete series can be stated as in 
Theorem 1.1 of \cite{muic-composition1series}.
\begin{theorem} \label{classification-discrete}
There exists a bijection between classes of discrete series $\sigma\in \text{Irr}'$ 
and all admissible triples $(\text{Jord},\sigma',\epsilon)$ 
denoted by
$$\sigma=\sigma_{(\text{Jord},\sigma',\epsilon)}$$
such that the following holds.
\begin{enumerate}
\item[$i)$] $\text{Jord}(\sigma)=\text{Jord}$ and $\sigma_{cusp}=\sigma'$.
\item[$ii)$] If a triple $(\text{Jord},\sigma',\epsilon)$ is alternated then
\begin{equation} \label{klasifikacija-strogih-tadic}
\sigma \hookrightarrow \prod_{i=1}^{m} \prod_{j=1}^{k_{\rho_i}}
\delta([\nu^{(\Phi_{\rho_i}(a_j^{\rho_i})+1)/2}\rho_i,\nu^{(a_j^{\rho_i}-1)/2}\rho_i])\rtimes \sigma_{cusp}
\end{equation}
is a unique irreducible subrepresentation,
 where $\{  \rho_1,...,\rho_m \}$ is a set of cuspidal representations appearing in Jord and every Jord$_{\rho_i}$ consists of 
${a^{\rho_i}_1 <\cdots < a^{\rho_i}_{k_{\rho_i}} }$.
\item[$iii)$] If $(a,\rho)\in \text{Jord}$ such that $(a_{-},\rho)\in \text{Jord}$
and $\epsilon(a,\rho)\epsilon(a_{-},\rho)^{-1}=1$ put 
\text{Jord}$''$=Jord$\setminus \{(a,\rho)(a_{-},\rho) \}$ and denote by
 $\epsilon''$ the restriction of $\epsilon$ on $\text{Jord}''$.
Then 
$$\sigma\hookrightarrow 
 \delta([\nu^{-(a_{-}-1)/2}\rho,\nu^{(a-1)/2}\rho])\rtimes 
\sigma_{(\text{Jord}'',\sigma',\epsilon'')}.
$$
Further, induced representation 
$
\delta([\nu^{-(a_{-}-1)/2}\rho,\nu^{(a_{-}-1)/2}\rho])\rtimes 
\sigma_{(\text{Jord}'',\sigma',\epsilon'')}
$ 
is a direct sum of of two non-equivalent representations $\tau_{\pm}$ and there
exist the unique $\tau\in \{ \tau_+, \tau_-\}$ such that
$$\sigma\hookrightarrow 
 \delta([\nu^{(a_{-}+1)/2}\rho,\nu^{(a-1)/2}\rho])\rtimes \tau.
$$
\end{enumerate}
\end{theorem}

 Given the correspondance we also denote $\epsilon$ by  $\epsilon_{\sigma}$. We provide more details on that function, see Theorem 1.3 of \cite{tadic-tempered}.

\begin{theorem} \label{tadic-klasifikacija}
Suppose that $(a,\rho)\in\text{Jord}$ and one of the following 
\begin{enumerate}
\item $a_ {-}$ is defined. Then 

$\epsilon_{\sigma}(a,\rho)\epsilon_{\sigma}(a_{-},\rho)^{-1}=1$ if and only if there exists a representation  $\pi'$ of some  $G_{n_{\pi'}}$ such that
\[
\sigma \hookrightarrow \delta([\nu^{(a_{-}+1)/2} \rho,\nu^{(a-1)/2} \rho])  \rtimes \pi'.
\]
%or equivalently 
%\[\delta([\nu^{(a_{-}+1)/2} \rho,\nu^{(a-1)/2} \rho])  \otimes \pi'\]
%is a subquotient in the appropriate Jacquet module of $\sigma$.

\item $a$ is even and $a=\text{min}(\text{Jord}_{\rho})$. Then

$\epsilon_{\sigma}(a,\rho)=1$ if and only if 
 there exists a representation  $\pi'$ of some  $G_{n_{\pi'}}$
   such that  
 \[ 
   \sigma \hookrightarrow \delta([\nu^{1/2} \rho,\nu^{(a-1)/2} \rho])  \rtimes \pi'.
\]
%or equivalently
%\[\delta([\nu^{1/2} \rho,\nu^{(a-1)/2} \rho])  \otimes \pi'\]
%is a subquotient in the appropriate Jacquet module of $\sigma$.
\item
$\rho\rtimes \sigma_{cusp}$ 
reduces and $a=\text{max}(\text{Jord}_{\rho}(\sigma))$.

Then there exist two irreducible nonequivalent tempered representations such that
$\rho\rtimes \sigma_{cusp}=\tau_{-1}\oplus \tau_{1}$. Here, a choice of index is made and we have the classification with respect to it. For any $k\in \mathbb{Z}_{>0}$ the representation $\delta([\nu\rho,\nu^k\rho])\rtimes \tau_i$, $i\in \{\pm 1\}$ has a unique irreducible subrepresentation denoted by
\[
\delta([\nu\rho,\nu^k\rho]_{\tau_i};\sigma_{cusp}).
\]
We have $\epsilon_{\sigma}(a,\rho)=i$ if and only if there exists an irreducible representation $\theta$ of $GL(m_{\theta},F)$ such that
\[
\sigma \hookrightarrow 
\theta \rtimes \delta([\nu\rho,\nu^{(a-1)/2}\rho]_{\tau_i};\sigma_{cusp}).
\]
\end{enumerate}
\end{theorem}

Discrete series $\sigma$ that correspond to alternated triples are called strongly positive discrete series. They can be characterized as follows (see Section 1 of \cite{moeglin}, Proposition 7 of \cite{tadic-diskretne} and Proposition  1.1 of  \cite{muic-composition1series}).
\begin{proposition} Let $\sigma \in \text{Irr } G_n$.
Then $\sigma$ is a discrete series that corresponds to the triple of alternated type 
if and only if  
for every embedding of form 
\[
\sigma \hookrightarrow \nu^{x_1}\rho'_1\times \cdots \times \nu^{x_k}\rho'_{k'} \rtimes \sigma'_{cusp}
\]
where $x_i\in \mathbb{R}$, $\rho_i \in \text{Irr }GL(m_{\rho_i},F)$ (this defines $m_{\rho_i}$) is  unitary cuspidal representation
$i=1,...,k'\in \mathbb{Z}_{>0}$  and  
$\sigma'_{cusp}\in \text{Irr }G_{n'}$ for some $n'$, is a cuspidal representation,
we have
\[
x_1>0,\ldots,x_{k'}>0.
\]
\end{proposition}

Now we want to prove a usefull fact about Jacquet modules of strongly positive discrete series. We note that they are calculated in  
\cite{matic-jacquet}.

\begin{proposition} \label{jacquet-modul-strogih}
Let $\sigma$ be a strongly positive representation and $2b+1<2c+1$  positive integers  of the same parity as numbers in Jord$_{\rho}$ for some 
$\rho$ appearing in Jord and 
$[2b+1,2c+1]\cap \text{Jord}_{\rho}=\emptyset$. Suppose that 
$\delta'\otimes \sigma'$ is an irreducible sumand in $\mu^*(\sigma)$ such that 
for some integer $2t+1\in [2b+1,2c+1]$, of the same parity as $2b+1$,
$\nu^t\rho$ is in cuspidal support of $\delta'$. 
Then there exists 
$a \in \text{Jord}_{\rho}$,  such that 
$c < (a-1)/2$ 
and 
$\nu^{(a-1)/2}\rho$ 
is in cuspidal support of $\delta'$.

\begin{proof}
We apply formula \eqref{komnozenje} on induced representation
in \eqref{klasifikacija-strogih-tadic}. Let $\rho=\rho_{i_0}$ for some $1\leq i_0 \leq m$.
To shorten the notation let us write 
$x_i^j=\frac{\Phi_{\rho_i}(a_j^{\rho_i})+1}{2}$  
and 
$y_i^j=\frac{a_j^{\rho_i}-1}{2}$
for 
$1\leq j\leq k_{\rho_i} , 1\leq i\leq m$.
Thus there exist indices
\[0 \leq s_i^j \leq r_i^j\leq 
 y_i^j - x_i^j +1, 
\text{ for }
1\leq j\leq k_{\rho_i} , 1\leq i\leq m, \]
such that
\begin{equation} \label{jacquet-modul-strogih-jednadzba}
\delta' \leq \prod_{i=1}^{m} \prod_{j=1}^{k_{\rho_i}}
\delta([\nu^{r_i^j-y_i^i}\rho_i,\nu^{-x_i^j}\rho_i])
\times
\delta([\nu^{y_i^j+1-s_i^j}\rho_i,\nu^{y_i^j}\rho_i]) \text{\ \ \  and }
\end{equation}
\begin{equation}
\sigma' \leq \delta([\nu^{y_i^j+1-r_i^j}\rho_i,\nu^{y_i^j-s_i^j}\rho_i])\rtimes \sigma_{cusp}.
\end{equation}
As all $x_i^j$ and $y_i^j$ are positive numbers and $\sigma$ is strongly positive discrete series, we have $r_i^j=y_i^j - x_i^j +1$ for all $1\leq j\leq k_{\rho_i} , 1\leq i\leq m$. 
So, there exist  $1\leq j_0 \leq k_{\rho_{i_0}}$ such that 
$s_{i_0}^{j_0}\geq 1$ and
$\nu^t \rho_{i_0}$ is in cuspidal support of 
$\delta([\nu^{y_{i_0}^{j_0}+1-s_{i_0}^{j_0}}\rho_{i_0},
\nu^{y_{i_0}^{j_0}}\rho_{i_0}])$. 
Now 
 $t\leq y_{i_0}^{j_0}$   and 
\eqref{jacquet-modul-strogih-jednadzba} implies that 
$\nu^{y_{i_0}^{j_0}}\rho_{i_0}$ is in cuspidal support of $\delta'$.
As $t \in [2b+1,2c+1]$ 
and
$[2b+1,2c+1]\cap \text{Jord}_{\rho_{i_0}}=\emptyset$
we have $c< y_{i_0}^{j_0}=(a_{j_0}^{\rho_{i_0}}-1)/2$.
 We take $a=a_{j_0}^{\rho_{i_0}}$.
\end{proof}
\end{proposition}
 
 Before we move to the class of induced representations that we consider, we  show our motivation.
 
\begin{proposition} \label{motivation}
Suppose that $\sigma$ is a discrete series, not strongly positive, such that
\begin{equation} \label{subquotients}
\sigma \hookrightarrow 
\delta(\Delta'_1)\times \cdots \times \delta(\Delta'_m) \rtimes \sigma_{sp}
\end{equation}
where $\Delta'_i=[\nu^{-d_i}\rho'_i,\nu^{e_i}\rho'_i]$, $\rho'_i$ is an unitarizable cuspidal representation of $GL(m_{\rho'_i},F)$
for $i=1,\ldots,m$, and the embedding is obtained using $iii)$ of Theorem 
\ref{classification-discrete} until we reach some strongly positive discrete series 
$\sigma_{sp}$. 

Then, either 
induced representations
$\delta(\Delta'_{i}) \times\delta(\Delta'_j)$ and
$\delta(\widetilde{\Delta'_i}) \times \delta(\Delta'_{j})$ 
are irreducible for all $1\leq i <j \leq m$ and we denote 
$\Delta_i=\Delta'_i$ for all $1 \leq i\leq m$, \\
or there exist a family of segments
$\Delta_i=[\nu^{-b_i}\rho_i,\nu^{c_i}\rho_i]$,
where $\rho_i$ is an unitarizable cuspidal representation of $GL(m_{\rho_i},F)$  
for  $i=1,\ldots,m$, such that
\begin{itemize}
\item 
we have equality of sets
\begin{equation} \label{edges}
\{ \nu^{d_i}\rho'_i,\nu^{e_i}\rho'_i | i=1,\ldots,m \}=
\{ \nu^{b_i}\rho_i,\nu^{c_i}\rho_i | i=1,\ldots,m \},
\end{equation}
\item 
$\delta(\Delta_{i}) \times\delta(\Delta_j)$ and
$\delta(\widetilde{\Delta_i}) \times \delta(\Delta_{j})$ 
are irreducible for all $1\leq i <j \leq m$,
\item 
in the appropriate Grothendieck group we have 
\[
\delta(\Delta_1)\times \cdots \times \delta(\Delta_m) \rtimes \sigma_{sp}
\leq
\delta(\Delta'_1)\times \cdots \times \delta(\Delta'_m) \rtimes \sigma_{sp}.
\]
\end{itemize}
Further, conditions $(C1)$ and $(C2)$ at the begining of Section \ref{sect-2}
are valid for the family $\{\Delta_1,\ldots,\Delta_m\}$ with respect to the $\sigma_{sp}$.
\end{proposition}

\begin{proof} 
By M\oe glin Tadi\'c classification of discrete series (C1) is valid for 
$\{\Delta'_1,\ldots,\Delta'_m \}$. 
If condition (C2) is not satisfied by the family 
$\{\Delta'_1,\ldots,\Delta'_m \}$, we construct family
$\{\Delta_1,\ldots,\Delta_m \}$ from $\{\Delta'_1,\ldots,\Delta'_k \}$
as follows.
Suppose that there exist $1\leq i <j \leq m$ such that 
$\Delta'_i$ and $\Delta'_j$ are linked. Then we replace them with  
$\Delta'_i \cup \Delta'_j$ and $\Delta'_i \cap \Delta'_j$ and possibly take contragredient to keep sum of exponents of edges of new segments positive.
Here $\Delta'_i \cap \Delta'_j\neq \emptyset$ by M\oe glin Tadi\'c classification of discrete series. 
%Observe also that for all $1\leq i <j \leq m$ if $\rho_i=\rho_j$ then $\Delta'_i \cup \Delta'_j$ is a segment, and this property has not changed.
The equation \eqref{edges} remained valid. It is not hard to check that  condition (C1) remained valid. The length of new induced representation, similar to one in \eqref{subquotients}, is smaller compared to one in  \eqref{subquotients}. Next, we do the same, on the newly obtained family, if 
there exist $1\leq i <j \leq m$ such that 
$\widetilde{\Delta'_i}$ and $\Delta'_j$ are linked.
We repeat these steps. As the induced representation in  \eqref{subquotients} is of finite length, the algorithm must stop.
Denote obtained family of segments by $\{\Delta_1,\ldots,\Delta_m\}$
as in the claim. 
\end{proof}

\section{Some discrete series extensions}
\label{sect-2}

In this section we introduce the notation that we use and 
provide some basic results about  extending given discrete series.

Let $\sigma_{sp}$ be a strongly positive discrete series of $G_n$, 
described by a triple $(\text{Jord},\sigma_{cusp},\epsilon)$. 
Let  
 $\mathcal{F}=\{\Delta_{i}=[\nu^{-b_i}\rho_i, \nu^{c_i}\rho_i] : i=1,\ldots,m \}$ be a family of segments such that
\begin{itemize}
\item[(C1)]
for every 
$1\leq i \leq m$, $\rho_i$ is an irreducible, selfdual, unitarizable and cuspidal  representation of 
$GL(m_{\rho_i},\mathbb{F})$ and one of the following holds:
\begin{itemize} 
\item[$\bullet$]
$\text{Jord}_{\rho_i}=\emptyset$,
$\nu^{\frac{1}{2}} \rho_i \rtimes {\sigma}_{cusp}$ reduces  and
$-\frac{1}{2}\leq b_i < c_i  \in \mathbb{Z}+\frac{1}{2}$,
\item[$\bullet$]  
$\text{Jord}_{\rho_i}=\emptyset$, $\rho_i \rtimes \sigma_{cusp}$ reduces and  
$0 \leq b_i < c_i \in \mathbb{Z}$,
\item[$\bullet$]  
$\text{Jord}_{\rho_i}\neq\emptyset$,
$0<2b_i +1< 2c_i +1$
are integers of the same parity as integers in $\text{Jord}_{\rho_i}$
and
$
\left[2b_i +1, 2c_i+1\right]\  \bigcap\ \text{Jord}_{\rho_i} =  \emptyset
$.
\end{itemize}
\item[(C2)]
induced representations
$\delta(\Delta_{i}) \times\delta(\Delta_j)$ and
$\delta(\widetilde{\Delta_i}) \times \delta(\Delta_{j})$ 
are irreducible for all $1\leq i <j \leq m$.
% for every two segments $\Delta,\Delta' \in \mathcal{F}$, we have either 
%$\Delta \cup \widetilde{\Delta}\subsetneq \Delta' \cap \widetilde{\Delta'}$ or
%$\Delta' \cup \widetilde{\Delta'}\subsetneq \Delta \cap \widetilde{\Delta}$.
\end{itemize}

We use $S$ and $Y$ to denote  two disjoint subsets of 
$\{1,\ldots,m\}$.

\begin{remark}
With respect to the class considered in Proposition \ref{motivation}
we added posibility of $-b_i=\frac{1}{2}$ for some $1\leq i \leq m$.
\end{remark}

\begin{lemma}
Suppose that there exist $1\leq i <j \leq m$ such that 
$\rho_i=\rho_j$. Then either $c_i<b_j$ or $c_j<b_i$.   
\end{lemma}
\begin{proof}
As $\Delta_i$ and $\Delta_j$ are not linked, but 
$\Delta_i  \cup \Delta_j$ is a segment, we have either 
$\Delta_i \subseteq  \Delta_j$ or 
$\Delta_j \subseteq  \Delta_i$.
The same goes for
$\widetilde{\Delta_i}$ and $\Delta_j$. 
%Since all borders of segments are different, 
Now simple case by case analysis gives either  $c_i<b_j$ or $c_j<b_i$ .
\end{proof}

Our  first step is to  describe irreducible subrepresentations of induced representations built from $\sigma_{sp}$ and segments that belong to family $\mathcal{F}$.

\begin{proposition}\  \label{prva}
Irreducible  subrepresentations 
\[
 \sigma \hookrightarrow  
 \prod_{i\in Y} \delta(\Delta_i) 
    \rtimes \sigma_{sp}
 \]
are discrete series representations obtained by extending $\sigma_{sp}$  
such that
\[
\begin{split}
\text{Jord}(\sigma)=\text{Jord}(\sigma_{sp})
\cup \{ (2b_i+1,\rho_i),(2c_i+1,\rho_i)| -b_i\neq 1/2,  i\in Y  
  \}
\\
\cup \{(2c_i+1,\rho_i)| -b_i= 1/2, i\in Y \}
\end{split}
\]
and
\[
\begin{split}
&\epsilon_{\sigma}(2b_i+1,\rho_i)\epsilon_{\sigma}(2c_i+1,\rho_i)^{-1}=1
 \text{ if } -b_i\neq 1/2,  i\in Y,
\\
&\epsilon_{\sigma}(2c_i+1,\rho)=1 \text{ if }  -b_i= 1/2,  i\in Y.
\end{split}
\]
These discrete series appear with multiplicity one in the induced representation. There are 
$2^l$
of them, where $l=\text{card}( \{ i | -b_i \neq 1/2, i\in Y\})$.
\end{proposition}

\begin{proof}
We extend $\sigma_{sp}$ using Theorems 2.1 and 2.3 of  
\cite{muic-composition1series}.
For every  
$i\in Y$ such that $-b_i=\frac{1}{2}$
 and we add $(2c_i+1, \rho_i)$ to the Jord$(\sigma_{sp})$ and extend 
$\epsilon$ function by value 1 on $(2c_i+1, \rho_1)$. After that, for every remaining segment $\Delta_i$, we add $\{(2b_i+1,\rho_i), (2c_i+1,\rho_i)\}$   to the set of Jordan blocks and we have two choices for extending the epsilon function.
 We have constructed 
$ 2^l $
 discrete series extensions of $\sigma_{sp}$. 
 They are all subrepresentations of 
\( \prod_{i \in Y}  \delta(\Delta_i) 
 \rtimes \sigma_{sp} \).

On the other hand
by Frobenius reciprocity every irreducible  subrepresentation of 
$ \prod_{i \in Y}\delta(\Delta_i) 
    \rtimes \sigma_{sp}$ contains
  $ \prod_{i \in Y}\delta(\Delta_i) 
    \otimes \sigma_{sp}$ 
   as an irreducible subquotient in the appropriate Jacuqet module. We will show that this subquotient occurs in 
   $\mu^*( \prod_{i \in Y}\delta(\Delta_i) 
    \rtimes \sigma_{sp})$  as many times as there are constructed extension of $\sigma_{sp}$.
This will immediately imply that irreducible subrepresentations of 
$\prod_{i \in Y}\delta(\Delta_i) 
    \rtimes \sigma_{sp}$ are precisely discrete series extensions that  we have constructed.

Using \eqref{komnozenje} we see that  
$\prod_{i \in Y}\delta(\Delta_i) \otimes \sigma_{sp}$
occurs in 
$\mu^*( \prod_{i \in Y}\delta(\Delta_i) \rtimes \sigma_{sp})$ if and only if there exist an irreducible representation  $\delta_1\otimes\sigma_1 \leq \mu^*(\sigma_{sp})$ and indices
$0\leq j_s \leq i_s \leq c_s + b_s +1$,
$s \in Y$ such that 
\begin{equation} \label{ulaganje-prva}
\prod_{ s \in Y} \delta([\nu^{-b_s}\rho_s,\nu^{c_s}\rho_s])
\leq
\prod_{ s\in Y}
\delta([\nu^{i_s-c_s}\rho_s,\nu^{b_s}\rho_s])
\times
\delta([\nu^{c_s+1-j_s}\rho_s,\nu^{c_s}\rho_s])\times \delta_1
\end{equation}
and
\begin{equation}
\sigma_{sp}\leq \prod_{s\in Y} 
\delta([\nu^{c_s+1-i_s}\rho_s,\nu^{c_s-j_s}\rho_s])\rtimes \sigma_1.
\end{equation}

We compare cuspidal support in \eqref{ulaganje-prva}. 
There exists $r \in Y$
such 
that for all  $ s \in Y$ if $s\neq r$ and  $\rho_s=\rho_r$
 we have $c_s<c_r$.
If $b_r=-\frac{1}{2}$, we have  
$\text{Jord}_{\rho_r}(\sigma_{sp})=\emptyset$, so
 $i_r=j_r=c_r+b_r+1=c_r-b_r$
If $b_r>0$ the representation $\delta_1$ can not contain in its cuspidal support  
$\nu^{-b_r}\rho_r$
because that would contradict strong positivity of $\sigma_{sp}$. Also, in this case, 
$\delta_1$ can not contain $\nu^{b_r+1}\rho_r$ in its cuspidal support, because 
$
\left[2b_r +1, 2c_r+1\right]\  \bigcap\ \text{Jord}_{\rho_r} =  \emptyset
$
and Proposition \ref{jacquet-modul-strogih} 
would imply existence of 
$\nu^x\rho_r$, $x>c_r$ in the cuspidal support of $\delta_1$. However, such can not be found on the left side of  \eqref{ulaganje-prva}. So $i_r=j_r=c_r+b_r+1$ or $i_r=j_r=c_r-b_r$.
We continue this step on $Y\setminus \{ r\}$. 
In the end, we have $\delta_1\otimes\sigma_1=1\otimes \sigma_{sp}$, appearing with multiplicity one in 
$\mu^*(\sigma_{sp})$. Thus, we have 
$2^l$
occurrences of 
$\prod_{i\in Y} \delta(\Delta_i) \otimes \sigma_{sp}$
in
$\mu^*( \prod_{i\in Y} \delta(\Delta_i) \rtimes \sigma_{sp})$.

We proved that subrepresentations of 
$ \prod_{i\in Y} \delta(\Delta_i) \rtimes \sigma_{sp}$ are precisely discrete series which are constructed as extensions of $\sigma_{sp}$.
\end{proof}

From now on we denote by $\sigma$ an irreducible subrepresentation as in 
Proposition \ref{prva}. This also includes case $\sigma=\sigma_{sp}$, for $Y=\emptyset$.
Our goal is to determine composition series of induced representation
\begin{equation} \label{class}
\prod_{i \in S} \delta(\Delta_{i})\rtimes \sigma.
\end{equation}
We proceed with a basic step
using some results about composition series of certain generalized principal  series obtained  in \cite{muic-composition1series}.

\begin{proposition}
\label{druga}
Let  
$\Delta=[\nu^{-b}\rho, \nu^{c}\rho]\in \mathcal{F}$, where  
$\Delta \neq \Delta_i$, for all $i \in Y$. 
 In the appropriate Grothendieck group we have
\[
\delta([\nu^{-b}\rho, \nu^{c}\rho])\rtimes \sigma=
\begin{cases}
\sigma_2 + Lang(\delta([\nu^{\frac{1}{2}}\rho, \nu^{c}\rho])\rtimes \sigma)   & \quad \text{ if }  -b=\frac{1}{2},                    \\ 
\sigma_3+ \sigma_4+Lang(\delta([\nu^{-b}\rho, \nu^{c}\rho])\rtimes \sigma) & \quad \text{else}.
\end{cases}
\]
Here $\sigma_2$ is a discrete series extension of $\sigma$ such that $Jord(\sigma_2)=
Jord(\sigma)\cup\{(2c+1,\rho)\}$ and $\epsilon_{\sigma_2}(2c+1,\rho)=1$ while
$\sigma_3$ and $\sigma_4$ are  non-isomorphic discrete series extensions of $\sigma$ 
such that $\text{Jord}(\sigma_2)=\text{Jord}(\sigma_3)=
\text{Jord}(\sigma)\cup \{(2b+1,\rho),(2c+1,\rho)\}$ and
$\epsilon_{\sigma_3}(2b+1,\rho)\epsilon_{\sigma_3}(2c+1,\rho)^{-1}=\epsilon_{\sigma_4}(2b+1,\rho)\epsilon_{\sigma_4}(2c+1,\rho)^{-1}=1$.
They appear as subrepresentations of the induced representation.
\end{proposition}

\begin{proof}
The second case follows directly from Theorem 2.1 of \cite{muic-composition1series}. 
So we consider the first case.
By the proof of Lemma 6.1 of \cite{muic-reducibility1principal} and the argument as in the proof of Theorem 2.1 of  \cite{muic-composition1series},   
$\delta([\nu^{\frac{1}{2}}\rho,\nu^c\rho])\rtimes \sigma$ reduces and all irreducible subquotients except $Lang(\delta([\nu^{\frac{1}{2}}\rho,\nu^c\rho])\rtimes \sigma)$
are discrete series subrepresentations. Also, their set of Jordan blocks is 
$\text{Jord}(\sigma)\cup \{(2c+1,\rho)\}$. 
We proceed by an induction over card$(Y)$.
The base case is covered by Theorem
2.3 of \cite{muic-composition1series}. Suppose that card$(Y)\geq 1$
and denote a minimal corresponding segment, with respect to the subset relation, by 
$[\nu^{-b_j}\rho,\nu^{c_j}\rho]$. 
Now 
\[
\sigma \hookrightarrow \sigma^+ \oplus \sigma^- \hookrightarrow
\delta([\nu^{-b_j}\rho,\nu^{c_j}\rho])\rtimes \sigma'
\]
where $\sigma'$ is a discrete series obtained from $\sigma$ by
removing $(2b_j+1,\rho)$ and 
$(2c_j+1,\rho)$ from Jord($\sigma$) and restricting the epsilon function.
Representations $\sigma^+$ and  $\sigma^-$  are non-equivalent discrete series extensions of  $\sigma'$ such that 
$\epsilon_{\sigma^+}(2b_j+1,\rho)=1$.
Let $\pi^r$ be a discrete series subrepresentation of 
$\delta([\nu^{\frac{1}{2}}\rho,\nu^c\rho])\rtimes \sigma^r$, where $r\in\{-,+\}$.
We want to prove that $\pi^r$ is an extension of $\sigma^r$ obtained by adding $(2c+1,\rho)$ to the set of Jordan blocks, and extending epsilon function by value 1, where $r\in\{\pm 1 \}$.
So 
\begin{align*} \label{embedding-of-two}
\pi^+\oplus \pi^-
&\hookrightarrow 
\delta([\nu^{\frac{1}{2}}\rho,\nu^c \rho])\rtimes \sigma^+ 
\oplus 
\delta([\nu^{\frac{1}{2}}\rho,\nu^c \rho])\rtimes \sigma^-
\\
&\cong
\delta([\nu^{\frac{1}{2}}\rho,\nu^c \rho])\rtimes (\sigma^+ \oplus \sigma^-)
\hookrightarrow
\delta([\nu^{\frac{1}{2}}\rho,\nu^c\rho])\times 
\delta([\nu^{-b_j}\rho,\nu^{c_j}\rho])\rtimes \sigma'
\\
&\cong 
\delta([\nu^{-b_j}\rho,\nu^{c_j}\rho])\times
\delta([\nu^{\frac{1}{2}}\rho,\nu^c\rho])\rtimes \sigma'.
\end{align*}
As $\sigma'$ can be embedded in an induced representation as in Proposition \ref{prva}, we conclude that $\pi^+\ncong\pi^-$. Further, 
there exist irreducible subquotients
$\sigma_0^{\pm} \leq \delta([\nu^{\frac{1}{2}}\rho,\nu^c\rho])\rtimes \sigma'$
such that
\[
\pi^{\pm} \hookrightarrow 
\delta([\nu^{-b_j}\rho,\nu^{c_j}\rho])\rtimes \sigma_0^{\pm}.
\]
By Remark 3.2 and Proposition 4.2 of \cite{moeglin} $\sigma_0^{\pm}$ is a discrete series.
Now an assumption of the induction implies that 
$\sigma _0^{\pm}  \hookrightarrow \delta([\nu^{\frac{1}{2}}\rho,\nu^c\rho])\rtimes \sigma'$ 
is an extension of $\sigma'$ obtained by adding $(2c+1,\rho)$ to the set of Jordan blocks and extending epsilon function by $1$. Thus, it does not depend on the choice of $\sigma^+$ or $\sigma^-$ and we simply denote it by $\sigma_0$.
So 
\[
\pi^+ \oplus \pi^- \hookrightarrow 
\delta([\nu^{-b_j}\rho,\nu^{c_j}\rho])\rtimes \sigma_0.
\]
To finish the proof it is enough to see that 
$\epsilon_{\pi^+}((2b_j+1,\rho))=1$.
Recall that 
$\pi^+ \hookrightarrow \delta([\nu^{\frac{1}{2}}\rho,\nu^c \rho])\rtimes \sigma^+$, $\epsilon_{\sigma^+}((2b_j+1,\rho))=\epsilon_{\sigma^+}((2c_j+1,\rho))=1$ and $2b_j+1=\text{min}(\text{Jord}_{\rho}(\sigma^+))$.
Theorem \ref{tadic-klasifikacija} implies that there exist an irreducible representation $\pi'$ such that
$\sigma^+\hookrightarrow \delta([\nu^{\frac{1}{2}}\rho,\nu^{b_j} \rho])\rtimes \pi'$. We have 
\begin{align*}
\pi^+ 
&\hookrightarrow 
\delta([\nu^{\frac{1}{2}}\rho,\nu^c \rho])\rtimes \sigma^+
\hookrightarrow
\delta([\nu^{\frac{1}{2}}\rho,\nu^c \rho])\times \delta([\nu^{\frac{1}{2}}\rho,\nu^{b_j} \rho])\rtimes \pi'
\\
&\cong
\delta([\nu^{\frac{1}{2}}\rho,\nu^{b_j} \rho]) \times
\delta([\nu^{\frac{1}{2}}\rho,\nu^c \rho])\rtimes \pi'.
\end{align*}
Again, there exists an irreducible representation $\pi''$ such that 
${\pi^+ 
\hookrightarrow 
\delta([\nu^{\frac{1}{2}}\rho,\nu^{b_j} \rho]) \rtimes \pi''}$.
So $\epsilon_{\pi^+}((2b_j+1,\rho))=1$ and the proof is finished.
\end{proof}

Finally, we are able to classify irreducible subrepresentations of \eqref{class} .

%%%%%%%%%%%%%%%%%%%%%%%%%%%%%%

\begin{proposition}  
\label{treca}
Irreducible  subrepresentations 
\[
 \sigma' \hookrightarrow  
 \prod_{i\in S} \delta(\Delta_i)     \rtimes \sigma
 \]
are discrete series representations obtained by extending $\sigma$  
such that
\[
\begin{split}
\text{Jord}(\sigma')=\text{Jord}(\sigma)
\cup \{ (2b_i+1,\rho_i),(2c_i+1,\rho_i)| -b_i\neq 1/2,  i\in S  
  \}
\\
\cup \{(2c_i+1,\rho_i) | -b_i= 1/2, i\in S  \}
\end{split}
\]
and
\[
\begin{split}
&\epsilon_{\sigma'}(2b_i+1,\rho_i)\epsilon_{\sigma'}(2c_i+1,\rho_i)^{-1}=1
 \text{ if } -b_i\neq 1/2,  i\in S,
\\
&\epsilon_{\sigma'}(2c_i+1,\rho)=1 \text{ if }  -b_i= 1/2,  i\in S.
\end{split}
\]
These discrete series appear with multiplicity one in the induced representation. There are 
$2^l$
of them, where $l=\text{card}( \{ i\ | -b_i \neq 1/2, i\in S\})$.
\end{proposition}

\begin{proof} 
We extend $\sigma$ using Proposition \ref{druga}. We pick elements of $S$ in an arbitrary manner.
If $i\in S$ is such that $-b_i= \frac{1}{2}$ we add $(2c_i+1, \rho_i)$ to the set of Jordan blocks and extend 
 the epsilon function by value 1 on $(2c_i+1, \rho_i)$. Else, we add $\{(2b_i+1,\rho_i), (2c_i+1,\rho_i)\}$   to the set of Jordan blocks and we have two choices for extending the epsilon function. We have constructed 
$2^l$
discrete series extensions of $\sigma$. They are all subrepresentations of
$\prod_{i\in S} \delta(\Delta_i) \rtimes \sigma$.

On the other hand
by Frobenius reciprocity every irreducible  subrepresentation of 
 $\prod_{i\in S} \delta(\Delta_i) \rtimes \sigma$ contains
 $\prod_{i\in S} \delta(\Delta_i) \otimes \sigma$ as an irreducible subquotient in the appropriate Jacquet module. We claim that 
 $\prod_{i\in S} \delta(\Delta_i) \otimes \sigma$ occurs in 
 $\mu^*( \prod_{i\in S} \delta(\Delta_i) \rtimes \sigma)$ as many times as there are constructed extensions.
This will immediately imply that subrepresentations of $\prod_{i\in S} \delta(\Delta_i) \rtimes \sigma$ are precisely discrete series which are constructed as extensions of $\sigma$.

It is enough to prove that 
$\prod_{i\in S} \delta(\Delta_i) \otimes \sigma$ occurs in 
$\mu^*( \prod_{i\in S\cup Y} \delta(\Delta_i) \rtimes \sigma_{sp})$ as many times as there are constructed extensions.
Using \eqref{komnozenje} we see that  
for every such occurence there exist an irreducible representation  
$\delta_1\otimes\sigma_1 \leq \mu^*(\sigma_{sp})$ and indices
$0\leq j_s \leq i_s \leq c_s + b_s +1$, $s\in Y\cup S$,  such that 
\begin{equation} \label{ulaganje-prva2}
\prod_{ s \in S} \delta([\nu^{-b_s}\rho_s,\nu^{c_s}\rho_s])
\leq
\prod_{ s \in Y \cup S}
\delta([\nu^{i_s-c_s}\rho_s,\nu^{b_s}\rho_s])
\times
\delta([\nu^{c_s+1-j_s}\rho_s,\nu^{c_s}\rho_s])\times \delta_1
\end{equation}
and
\begin{equation} \label{ulaganje-druga2}
\sigma\leq \prod_{s \in Y \cup S} \delta([\nu^{c_s+1-i_s}\rho_s,\nu^{c_s-j_s}\rho_s])\rtimes \sigma_1.
\end{equation}
We compare cuspidal support in \eqref{ulaganje-prva2}.

There exists $r \in Y\cup S$
such 
that for all  $ s \in (Y\cup S)\setminus \{r \} $ if  $\rho_s=\rho_r$
 we have $c_s<c_r$. First suppose that $r\in S$.
If $b_r=-\frac{1}{2}$,  
$\text{Jord}_{\rho_r}(\sigma_{sp})=\emptyset$, so
 $i_r=j_r=c_r+b_r+1=c_r-b_r$.
If $b_r>0$ the representation $\delta_1$ can not contain in its cuspidal support  
$\nu^{-b_r}\rho_r$
because that would contradict strong positivity of $\sigma_{sp}$. Also, in this case, 
$\delta_1$ can not contain $\nu^{b_r+1}\rho_r$ in its cuspidal support, because 
$
\left[2b_r +1, 2c_r+1\right]\  \bigcap\ \text{Jord}_{\rho_r}(\sigma_{sp}) =  \emptyset
$
and Proposition \ref{jacquet-modul-strogih} 
would imply existence of 
$\nu^x\rho_r$, $x>c_r$ in the cuspidal support of $\delta_1$. However, such can not be found on the left side of 
of  \eqref{ulaganje-prva2}. So $i_r=j_r=c_r+b_r+1$ or $i_r=j_r=c_r-b_r$.
Now suppose that $r \in Y$.  
In \eqref{ulaganje-prva2} $\nu^{-b_r}\rho_r$ and $\nu^{c_r}\rho_r$ do not appear in the cuspidal support on the left hand side of the inequality, so they can not appear on the right hand side either. Thus 
$i_r=c_r+b_r+1$ and $j_r=0$.
We continue the above procedure on $(Y\cup S)\setminus \{ r\}$. 
In the end we have $\delta_1=1$, so $\sigma_1=\sigma_{sp}$ and 
\eqref{ulaganje-druga2} looks like
\[
\sigma\leq \prod_{s \in Y} \delta([\nu^{-b_s}\rho_s,\nu^{c_s}\rho_s])\rtimes \sigma_{sp},
\]
 but $\sigma$ occurs here with multiplicity one by Proposition \ref{prva}.
 Thus, we have $2^l$ occurrences of 
$\prod_{i\in S} \delta(\Delta_i) \otimes \sigma$ in 
$\mu^*( \prod_{i\in Y\cup S} \delta(\Delta_i) \rtimes \sigma_{sp})$.

 We proved that irreducible subrepresentations of 
 $\prod_{i\in S} \delta(\Delta_i) \rtimes \sigma$ are precisely discrete series which are constructed as extensions of $\sigma$.

\end{proof}

\section{The main theorem}
\label{sect-3}
In this section we provide composition series of considered representations. One should keep in mind Proposition \ref{treca} and the  notation introduced in Section \ref{sect-2}.  

\begin{theorem} \label{glavni}
Let $\sigma$ be a discrete series as in Proposition \ref{prva}. 
 In the appropriate Grothendieck group we have
\begin{equation} \label{glavna1}
\displaystyle
\prod_{i\in S}
 \delta(\Delta_i) 
    \rtimes \sigma
=
\sum_{X\subseteq S} \ \ \ 
\sum_{ \sigma'\hookrightarrow \prod_{j\in S\setminus X} \delta(\Delta_j)\rtimes \sigma}
Lang(\prod_{i\in X} \delta(\Delta_i)\rtimes \sigma')
\end{equation}
where  $\sigma'$ is used to denote an irreducible representation. Let 
$k=\text{card}(S)$.
For every integer $0\leq l \leq k$ let $V_l$ be an image of the intertwining operator
\[
\bigoplus_{\substack{X\subseteq S,\\ \text{card}(X)=l}} \ \ 
\bigoplus_{ \sigma'\hookrightarrow \prod_{j\in S\setminus X} \delta(\Delta_j)\rtimes \sigma} \ 
\left(\prod_{i\in X} \delta(\Delta_i)\rtimes \sigma' \right)
\longrightarrow
\prod_{i\in S} \delta(\Delta_i)\rtimes \sigma
\]
given by $(x_1,x_2,\ldots) \mapsto x_1+x_2+\cdots$. Thus 
$V:=V_{k}=\prod_{i\in S} \delta(\Delta_i)\rtimes \sigma$ and  
$V_0$ is a direct sum of  irreducible subrepresentations of 
$V_k$ which are described by Proposition \ref{treca}. 
Further, we have a filtration
$\{0\}=V_{-1}\subseteq V_0\subseteq \ldots  \subseteq V_k$, 
where for every integer $0\leq l \leq k$ we have
\begin{equation} \label{glavna2}
V_l/V_{l-1}\cong
\bigoplus_{\substack{X\subseteq S,\\ \text{card}(X)=l}} \ \ 
\bigoplus_{ \sigma'\hookrightarrow \prod_{j\in S\setminus X} \delta(\Delta_j)\rtimes \sigma} \ 
Lang(\prod_{i\in X} \delta(\Delta_i)\rtimes \sigma' ).
\end{equation}
\end{theorem}

\begin{proof}

First we prove  formula \eqref{glavna1} by an induction over $k$. Case 
$k=1$ is  Proposition \ref{druga}. So we assume that $k>1$ and the formula is valid for strictly smaller cardinalities. As formulas are invariant to permutations of integers we also assume $S=\{1,\ldots,k\}$ and
 \[
 e(\Delta_1)\geq e(\Delta_2) \geq \cdots \geq e(\Delta_k )>0.
\]
 Consider a  standard representation and the first set of non-trivial intertwinings
\begin{align*}
\delta(\Delta_1)  \times \delta(\Delta_2) \times  \cdots \times \delta(\Delta_k)   \rtimes \sigma 
\cong \\
\delta(\Delta_2)  \times \delta(\Delta_1) \times \cdots \times \delta(\Delta_k)   \rtimes \sigma 
\cong \\
\delta(\Delta_2)  \times \cdots \times   \delta(\Delta_k)
\times \delta(\Delta_1)   \rtimes \sigma 
\rightarrow \\
\delta(\Delta_2)  \times \cdots \times   \delta(\Delta_2) \times 
\delta(\widetilde{\Delta}_1)   \rtimes \sigma 
\cong \\
\delta(\Delta_2)  \times \cdots \times  \delta(\widetilde{ \Delta}_1) \times \delta(\Delta_k)     \rtimes \sigma 
\cong \\
\delta(\widetilde{ \Delta}_1 ) \times \delta(\Delta_2)  \times \cdots \times  \delta(\Delta_k)    \rtimes \sigma \ \ \ 
\end{align*}
And so on, with the last line of the last set of intertwinings being
\[
\delta(\widetilde{ \Delta}_1 ) \times \delta(\widetilde{\Delta}_2) \times \cdots  \times   \delta(\widetilde{\Delta}_k)   \rtimes \sigma.
\]

Let  $K_l$ be the semisimplification of the kernel of the non-isomorphism in the  $l$-th set of intertwinings, where $1 \leq l\leq k$.
 By Proposition \ref{druga}, in the appropriate Grothendieck group, we have   
\begin{equation} \label{definicija-jezgre}
K_{l}=
\sum_{ \sigma^{l}\hookrightarrow  \delta(\Delta_l)\rtimes \sigma} \quad 
\prod_{i\in S\setminus \{ l \}} \delta(\Delta_i)\rtimes \sigma^{l}
\end{equation}
where $\sigma^{l}$ denotes an irreducible subrepresentation.
By the assumption of the induction, this is equal to
\begin{align*}
\sum_{ \sigma^{l}\hookrightarrow  \delta(\Delta_l)\rtimes \sigma} \quad
\sum_{X\subseteq S\setminus \{l \}} \ \ \ 
\sum_{ \sigma'\hookrightarrow \prod_{j\in (S\setminus \{ l \} )\setminus X}
 \delta(\Delta_j)\rtimes \sigma^{l}}
Lang(\prod_{i\in  X} \delta(\Delta_i)\rtimes \sigma') = \\
\sum_{X\subseteq S\setminus \{l \}} \quad
\sum_{ \sigma^{l}\hookrightarrow  \delta(\Delta_l)\rtimes \sigma} \quad 
\sum_{ \sigma'\hookrightarrow \prod_{j\in S\setminus ( X \cup \{ l \} )} \delta(\Delta_j)\rtimes \sigma^{l}}
Lang(\prod_{i\in X } \delta(\Delta_i)\rtimes \sigma').
\end{align*}
By the proof of Proposition \ref{treca}
 the above expression is equal to
\begin{equation} \label{jezgre}
K_{l}=
\sum_{ X \subseteq S \setminus \{ l\} } \quad
\sum_{ \sigma' \hookrightarrow \prod_{j \in S \setminus X  }
 \delta(\Delta_j)\rtimes \sigma} \quad
 Lang( \prod_{i\in X } \delta(\Delta_i)\rtimes \sigma' ).
\end{equation}
Here for $X=\emptyset$ we get all discrete series subrepresentations of 
${\prod_{i\in S} \delta(\Delta_i)\rtimes \sigma}$.
These discrete series subrepresentations appear with multiplicity one by Proposition \ref{treca}. For $X \neq \emptyset$, we claim that every $Lang( \prod_{i\in X } \delta(\Delta_i)\rtimes \sigma' )$ appears with multiplicity one in ${\prod_{i\in S} \delta(\Delta_i)\rtimes \sigma}$.
Since 
\[
Lang( \prod_{i\in X } \delta(\Delta_i)\rtimes \sigma' )
\hookrightarrow \prod_{i\in X } \delta(\widetilde{\Delta}_i)\rtimes \sigma',
\] 
using Frobenius reciprocity, we have 
\[
\mu^*(
Lang( \prod_{i\in  X } \delta(\Delta_i)\rtimes \sigma' ) )
\geq
\prod_{i\in X } \delta(\widetilde{\Delta}_i)\otimes \sigma'.
\]
It is enough to prove that 
$\prod_{i\in X } \delta(\widetilde{\Delta}_i)\otimes \sigma'$
 appears in $\mu^*( \prod_{s\in S\cup Y} \delta(\Delta_i) \rtimes \sigma_{sp})$ with multiplicity one.
Using \eqref{komnozenje}, for every such  occurrence there exist an irreducible representation  $\delta_1\otimes\sigma_1 \leq \mu^*(\sigma_{sp})$ and indices
$0\leq j_s \leq i_s \leq c_s + b_s +1$, 
$s\in Y\cup S$ such that 
\begin{equation} \label{ulaganje-novo2}
\prod_{i \in  X} \delta([\nu^{-c_i}\rho_s,\nu^{b_i}\rho_s])
\leq
\prod_{ s \in Y \cup S}
\delta([\nu^{i_s-c_s}\rho_s,\nu^{b_s}\rho_s])
\times
\delta([\nu^{c_s+1-j_s}\rho_s,\nu^{c_s}\rho_s])\times \delta_1
\end{equation}
and
\begin{equation} \label{ulaganje-novo3}
\sigma'\leq \prod_{s\in Y \cup S} \delta([\nu^{c_s+1-i_s}\rho_s,\nu^{c_s-j_s}\rho_s])\rtimes \sigma_1.
\end{equation}

There exists $r \in Y\cup S$
such 
that for all  $ s \in (Y\cup S)\setminus \{r \} $ if  $\rho_s=\rho_r$
 we have $c_s<c_r$. 
 First suppose that $r\in Y$ or $r\in S \setminus X$.  
 On the left hand side neither $\nu^{b_r}\rho_r$ nor $\nu^{c_r}\rho_r$ can appear.
 So $i_r=c_r+b_r+1$ and $j_r=0$.
Now suppose that $r \in  X$.
 On the left hand side $\nu^{c_r}\rho_r$ does not appear so $j_r=0$.
 As $\sigma_{sp}$ is strongly positive,
 $\delta_1$ can not contain  $\nu^{-c_r}\rho_r$ in its cuspidal support. So we have $i_r=0$.
We continue above procedure on $(Y\cup S)\setminus \{ r\}$. 
In the end $\delta_1=1$, $\sigma_1=\sigma_{sp}$ and
 \eqref{ulaganje-novo3} looks like
\[
\sigma'\leq \prod_{s\in Y\cup (S\setminus X)} \delta([\nu^{-b_s}\rho_s,\nu^{c_s}\rho_s])\rtimes \sigma_1
\]
 but this occurs once by Proposition \ref{prva}.

Now using \eqref{jezgre} we sum $K_l$ over $1 \leq l \leq k$ and count irreducible summands once to get 
\begin{equation}\label{formul}
\sum_{ X \subsetneq S } \quad
\sum_{ \sigma' \hookrightarrow \prod_{j \in X \setminus S }
 \delta(\Delta_j)\rtimes \sigma} \quad
 Lang( \prod_{i\in  X } \delta(\Delta_i)\rtimes \sigma' )
\end{equation}  
 as the semisimplification of the kernel of composition of all above intertwinings.
This composition  is a non-trivial intertwining operator
\[
\delta(\Delta_1)  \times \delta(\Delta_2) \times  \cdots \times \delta(\Delta_k)   \rtimes \sigma
\rightarrow
\delta(\widetilde{ \Delta}_1 ) \times \delta(\widetilde{\Delta}_2) \times \cdots  \times   \delta(\widetilde{\Delta}_k)   \rtimes \sigma
\]
whose image is  
\( \displaystyle Lang( \prod_{i\in S } \delta(\Delta_i)\rtimes \sigma ) \). Formula \eqref{glavna1} of the theorem is obtained when one writes
the image 
as 
\[ \displaystyle
\sum_{ X = S } \quad
\sum_{ \sigma' \hookrightarrow \prod_{j \in S \setminus X}
 \delta(\Delta_j)\rtimes \sigma} \quad
 Lang( \prod_{i\in X } \delta(\Delta_i)\rtimes \sigma' )
\]  
and adds it to \eqref{formul}.

Finally we prove \eqref{glavna2}. Since 
$\prod_{i\in S } \delta(\Delta_i)\rtimes \sigma$ is a multiplicity one representation, for $1\leq  l\leq k$ we apply formula \eqref{glavna1} on 
\[
\bigoplus_{\substack{X\subseteq S,\\ \text{card}(X)=l}} \ \ 
\bigoplus_{ \sigma'\hookrightarrow \prod_{j\in S\setminus X} \delta(\Delta_j)\rtimes \sigma} \ 
\left(\prod_{i\in X} \delta(\Delta_i)\rtimes \sigma' \right)
\]
to obtain 
\[
V_l=\bigoplus_{\substack{X\subseteq S,\\ \text{card}(X)\leq l}} \ \ 
\bigoplus_{ \sigma'\hookrightarrow \prod_{j\in S\setminus X} \delta(\Delta_j)\rtimes \sigma} \ 
Lang(\prod_{i\in X} \delta(\Delta_i)\rtimes \sigma' )
\]
in the appropriate Grothendieck group. So $V_{l-1}\subseteq V_l$ and in the Grothendieck group we have
\[
V_l/V_{l-1}=\bigoplus_{\substack{X\subseteq S,\\ \text{card}(X)= l}} \ \ 
\bigoplus_{ \sigma'\hookrightarrow \prod_{j\in S\setminus X} \delta(\Delta_j)\rtimes \sigma} \ 
Lang(\prod_{i\in X} \delta(\Delta_i)\rtimes \sigma' )
\]
Given $X\subseteq S$, $\text{card}(X)= l$ and an irreducible representation
${\sigma'\hookrightarrow \prod_{j\in S\setminus X} \delta(\Delta_j)\rtimes \sigma }$
we have 
\[
\prod_{i\in X} \delta(\Delta_i)\rtimes \sigma'  \hookrightarrow V_l.
\]
By formula \eqref{glavna1}, in the Grothendieck group we have 
\[
\begin{split}
\prod_{i\in X} \delta(\Delta_i)\rtimes \sigma'  =
&
Lang(\prod_{i\in X} \delta(\Delta_i)\rtimes \sigma')+
\\
&
\sum_{X'\subsetneq X} \ \ \ 
\sum_{ \sigma''\hookrightarrow \prod_{j\in X\setminus X'} \delta(\Delta_j)\rtimes \sigma'}
Lang(\prod_{i\in X'} \delta(\Delta_i)\rtimes \sigma'')
\end{split}
\]
where only the first summand is not an irrreducible subquotient of $V_{l-1}$. So
\[
Lang(\prod_{i\in X} \delta(\Delta_i)\rtimes \sigma') \hookrightarrow V_l/V_{l-1}.
\]
Formula \eqref{glavna2} follows. 
\end{proof}

Using notation as in Theorem \ref{glavni} we have

\begin{corollary}
	Let 
	$\pi \leq \prod_{i\in S}
	\delta(\Delta_i) 
	\rtimes \sigma$ 
	be any irreducible subquotient. Taking possible contragredients of  segments in  
	$\{\Delta_i : i\in S\}$, there exist segments
	${\{\Delta'_i : i\in S\}}$, such that
	\[
	\pi \hookrightarrow
	\prod_{i\in S}
	\delta(\Delta'_i) 
	\rtimes \sigma.
	\]
\end{corollary}

\begin{corollary}
Suppose that $S=S_1\cup S_2$, $S_1\cap S_2=\emptyset$, $S_1, S_2 \neq \emptyset$. In the appropriate Grothendieck group we have
\begin{equation} \label{glavna3}
\begin{split}
\prod_{i\in S_1} &
\delta(\Delta_i)
\rtimes
Lang( \prod_{j\in S_2}\delta(\Delta_j)\rtimes \sigma)=
\\
&\sum_{X\subseteq S_1} \ \ \ 
\sum_{ \sigma'\hookrightarrow \prod_{j\in S_1\setminus X} 
\delta(\Delta_j)\rtimes \sigma}
Lang(\prod_{i\in X \cup S_2} \delta(\Delta_i)\rtimes \sigma')
\end{split}
\end{equation}
where  $\sigma'$ is used to denote an irreducible representation. 
Moreover, the induced representation has the filtration $\{0\}=W_{-1}\subseteq W_0\subseteq \ldots  \subseteq W_{\text{card}(S_1)}$, 
and for $l\in \{0,\ldots, \text{card}(S_1)\} $
\begin{equation} \label{glavna4}
W_l/W_{l-1}\cong
\bigoplus_{\substack{X\subseteq S_1,\\ \text{card}(X)=k}} \ \ 
\bigoplus_{ \sigma'\hookrightarrow \prod_{j\in S_1\setminus X} \delta(\Delta_j)\rtimes \sigma} \ 
Lang(\prod_{i\in X\cup S_2} \delta(\Delta_i)\rtimes \sigma' ).
\end{equation}
\end{corollary}

\begin{proof}
Using intertwining operators as in the proof of Theorem \ref{glavni}, we see that the induced representation in \eqref{glavna3} is a homomorphic image of 
\begin{equation} \label{velika-inducirana2}
V=\prod_{i\in S} \delta(\Delta_i)\rtimes \sigma
 \end{equation}
with the kernel $D$ in the appropriate Grothendieck group being sum of semisimplifications of
\begin{equation} \label{dok-kor-jed2}
\sum_{ \sigma' \hookrightarrow \delta({\displaystyle
 \Delta_{ s_2}})\rtimes \sigma} \quad 
\prod_{i\in S\setminus \{ s_2\}} \delta(\Delta_i)\rtimes \sigma'
\end{equation}
where $s_2$ runs over $S_2$ and we take different irreducible sumands once. By \eqref{definicija-jezgre} and \eqref{jezgre} we know that \eqref{dok-kor-jed2}
is equal to 
\[
\sum_{ X \subseteq S \setminus \{ s_2 \} } \quad
\sum_{ \sigma' \hookrightarrow \prod_{j \in S \setminus X  }
 \delta(\Delta_j)\rtimes \sigma} \quad
 Lang( \prod_{i\in X } \delta(\Delta_i)\rtimes \sigma' ).
\]
Removing these irreducible subquotients from \eqref{glavna1} gives us \eqref{glavna3}.

Now we prove \eqref{glavna4}. For  $l\in\{0,...,\text{card}(S) \}$ we have an epimorphism
\[
V_l/V_{l-1} \cong 
^{\textstyle{V_l/(V_{l-1}\cap D)}} /_{\textstyle{V_{l-1}/(V_{l-1}\cap D)}} 
\longrightarrow 
 ^{\textstyle{V_l/(V_{l}\cap D)}} /_{\textstyle{V_{l-1}/(V_{l-1}\cap D)}}.
\]
As spaces 
$V_l/(V_{l}\cap D), l=-1,...,\text{card} (S),$ provide filtration for $V/D$
we use \eqref{glavna2} and \eqref{glavna3} to obtain \eqref{glavna4}. 
\end{proof}

Finally, we consider M\oe glin Tadi\'c classification of discrete series

\begin{corollary} \label{mt-dekompozicija}
	Let $\sigma$ be any discrete series described by    
	M\oe glin Tadi\'c classification
	as a subrepresentation of the induced representation 
	\[
	\sigma \hookrightarrow
	\delta(\Delta_{1})\times\cdots\times \delta(\Delta_k)\rtimes \sigma_{sp},
	\]
	where $\sigma_{sp}$ is a strongly positive discrete series and $\Delta_{1},\ldots ,\Delta_{k}$ are segments.
	Assume that for all $i\neq j$ induced representations
	$\delta(\Delta_{i}) \times\delta(\Delta_j)$ and
	$\delta(\widetilde{\Delta_i}) \times \delta(\Delta_{j})$ 
	are irreducible. Then Theorem \ref{glavni} applies to the induced representation. The induced representation has $2^k$ irreducible subrepresentations. They are discrete series extensions of $\sigma_{sp}$.
\end{corollary}
\begin{proof}
	Follows from the Theorem \ref{glavni}, Proposition \ref{motivation} and Proposition \ref{treca}. 
\end{proof}

%\Acknowledgements{If necessary your acknowledgements enter here.}
%\newpage
%\begin{thebibliography}{M{\oe}gTad02}


\begin{thebibliography}{}
%\bibitem[BerZel77]{bernstein-zelevinsky-ind-repns-I}
\bibitem{bernstein-zelevinsky-ind-repns-I}
I.~N. Bernstein and A.~V. Zelevinsky, \emph{Induced representations of
  reductive {$p$}-adic groups. {I}}, Ann. Sci. \'{E}cole Norm. Sup. (4)
  10 (1977), no.~4, 441-472.


%\bibitem[Cig18]{ciganovic}
\bibitem{ciganovic}
I.~Ciganovi\'c, \emph{Composition series of a class of induced representations, a case of one half cuspidal reducibility}, Pacific J. Math. 296 (2018),  no. 1, 21-30.

%\bibitem[Mat16]{ba}
\bibitem{ba}
I.~Mati\'c, \emph{First occurrence indices of tempered representations of metaplectic groups,} Proc. Amer. Math. Soc. 144 (2016), no. 7, 3157-3172.



%\bibitem[Mat13]{matic-jacquet}
\bibitem{matic-jacquet}
\bysame , \emph{Jacquet modules of strongly positive representations of the metaplectic group $\widetilde{Sp(n)}$}, Trans. Amer. Math. Soc. 365 (2013), no. 5, 2755-2778.


%\bibitem[M{\oe}g02]{moeglin}
\bibitem{moeglin}
C.~M{\oe}glin, \emph{Sur la classification des s\'{e}ries discr\`{e}tes des groupes classiques $p$-adiques: param\`{e}tres de Langlands et exhaustivit\'{e}. (French) [On the classification of discrete series of classical p-adic groups: Langlands parameters and completeness]} J. Eur. Math. Soc. (JEMS) 4 (2002), no. 2, 143-200.




%\bibitem[M{\oe}gTad02]{tadic-diskretne}
\bibitem{tadic-diskretne}
\bysame, M.~Tadi\'{c}, 
\textit{Construction of discrete series for classical $p$-adic groups,}
J. Amer. Math. Soc. 15 (2002), no. 3, 715-786. 



%\bibitem[Mui04]{muic-composition1series}
\bibitem{muic-composition1series}
G.~Mui{\'c}, \textit{Composition series of generalized principal series; the case of strongly positive discrete series}, Israel J. Math. 140 (2004), 157-202. 




%\bibitem[Mui05]{muic-reducibility1principal}
\bibitem{muic-reducibility1principal}
\bysame, \textit{Reducibility of generalized principal series}, Canad. J. Math. 57 (2005), no. 3, 616-647.
 
 
 
%\bibitem[Tad95]{tadic-structure}
\bibitem{tadic-structure}
M.~Tadi{\'c}, \textit{Structure arising from induction and Jacquet modules of representations of classical $p$-adic groups}, J. Algebra 177 (1995), no. 1, 1-33.

 
 %\bibitem[Tad02]{tadic-family}
 \bibitem{tadic-family}
\bysame,
\textit{A family of square integrable representations of classical $p$-adic groups in the case of general half-integral reducibilities} Glas. Mat. Ser. III 37(57) (2002), no. 1, 21-57.

 
%\bibitem{tadic-red-par-ind}
%\bysame, \textit{On reducibility of parabolic induction}, Israel J. Math.
%  \textbf{107} (1998), 29-91.

%\bibitem[Tad13]{tadic-tempered} 
\bibitem{tadic-tempered}
\bysame, \textit{On tempered and square integrable representations of classical p-adic groups}, Sci. China Math. 56 (2013), no. 11, 2273-2313.



%\bibitem[Zel80]{zelevinsky-ind-repns-II}
\bibitem{zelevinsky-ind-repns-II}
A.~V. Zelevinsky, \textit{Induced representations of reductive $p$-adic groups.
  {II}. On irreducible representations of $GL(n)$}, Ann. Sci. \'{E}cole
  Norm. Sup. (4) \textbf{13} (1980), no.~2, 165-210.


\end{thebibliography}
\end{document}